\documentclass[11pt]{amsart}
\usepackage{epsfig, graphics, psfrag, color}
\usepackage[small,nohug,heads=littlevee]{diagrams}
\diagramstyle[labelstyle=\scriptstyle]
\usepackage[margin=0.9in]{geometry}

\newcommand{\RR}{{\mathbb{R}}}

\theoremstyle{plain}
\newtheorem{theorem}{Theorem}[section]

\newtheorem{lemma}[theorem]{Lemma}

\newtheorem{proposition}[theorem]{Proposition}

\newtheorem*{namedtheorem}{\theoremname}
\newcommand{\theoremname}{testing}

\theoremstyle{definition}

\psfrag{r}[][]{$r_1$}

\author{Alexander Coward}
\author{Joel Hass}

\begin{document}

\title[Topological and physical link theory are distinct]
{Topological and physical link theory are distinct}

\maketitle
\vspace{-1cm}
\begin{abstract}
Physical knots and links are one-dimensional submanifolds of $\RR^3$ with fixed length and thickness.  We show that isotopy classes in this category can differ from those of classical knot and link theory. In particular we exhibit a Gordian Split Link, a two component link that is split in the classical theory but cannot be split with a physical isotopy.
\end{abstract}

\section{Introduction}

The theory of knots and links studies one-dimensional submanifolds of $\RR^3$. These are often described as loops of string, or rope, with their ends glued together.  Real ropes  however are not one-dimensional, but have a positive thickness and a finite length. Indeed, most physical applications of knot theory are related more closely to the theory of  knots of fixed thickness and length than to classical knot theory. For example, biologists are interested in curves of fixed thickness and length as a model for DNA and protein molecules. In these applications the thickness of the curve modeling the molecule  plays an essential role in determining the possible configurations. 

In this paper we show that  the equivalence class of a link in $\RR^3$ under an isotopy that preserves thickness and length can be distinct from the classical equivalence class under isotopy. We thus show for the first time that the theory of  physically realistic curves of fixed thickness and length is distinct from the classical theory of knots and links.

The two most fundamental problems concerning physical knots and links are to show the existence of a {\em Gordian Unknot} and a  {\em Gordian Split Link}.  A Gordian Unknot is a loop of fixed thickness and length whose core is unknotted, but which cannot be deformed to a round circle by an isotopy fixing its length and thickness.  A Gordian Split Link is a pair of loops of fixed thickness whose core curves can be split, or isotoped so that its two components are separated by a plane, but cannot be split by an isotopy fixing each component's length and thickness.
In this paper we establish the existence of  such a link.

\begin{theorem}\label{maintheorem}
A Gordian Split Link exists.
\end{theorem}

The proof of Theorem \ref{maintheorem} is by a construction of a link, illustrated in Figure \ref{link}, that can be topologically but not physically split. 

\begin{figure}[h!]
\centering
\includegraphics[width=.23\textwidth]{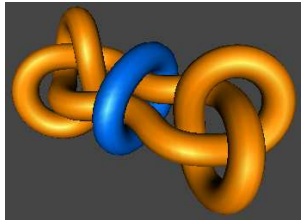}
\caption{A Gordian Split Link.} \label{link}
\end{figure}  

There has been extensive investigation, much of it experimental, into the properties of shortest representatives of physical knots, called {\em ideal knots}, and into the possible existence of Gordian Unknots \cite{BS,CKS,CKS2,DEJ,GM, P2}.  A candidate for a  Gordian Unknot appeared in work of  Freedman, He and Wang \cite{FHW}, who studied energies associated to curves in $\RR^3$ and associated gradient flows \cite{FHW}. This curve was studied numerically by Piera{\'n}ski  \cite{P}, who developed a computer program called SONO (Shrink On No Overlaps) to numerically shorten a curve of fixed thickness while avoiding overlaps. The program unexpectedly succeeded in unraveling the Freedman-He-Wang example. However there are more complicated examples that do fail to unravel under SONO, and hence give numerical evidence for  the existence of Gordian Unknots.  A proof of the existence of Gordian Unknots or Gordian Split Links based on a rigorous analysis of such algorithms is plausible, but has not yet been found.
  
We now give precise definitions. 
We say that a knot or link $L$ in $\RR^3$ is \emph{$r$-thick} if it is differentiable  and its open radius-$r$ normal disk bundle is embedded. This means that the collection of flat, radius-$r$ two-disks intersecting  $L$ perpendicularly at their centers have mutually disjoint interiors. An isotopy of a knot or link maintaining $r$-thickness throughout is called an \emph{$r$-thick isotopy}. By rescaling we may take  $r=1$, and take \emph{thick} to mean 1-thick.
A \emph{physical isotopy} is a thick isotopy of a knot or link that preserves the length of each component. 

Theorem \ref{maintheorem} is proved by an explicit construction of a two-component thick link $L$ that is split but admits no physical isotopy splitting its components. To construct this link  we begin by placing two points $A$ and $B$ at $(1,0,0)$ and $(-1,0,0)$.  Let $AB$ denote the straight line between these two points. The first component $L_1$ of $L$ is any thick curve encircling $AB$ in the plane $z = 0$, disjoint from the open, radius 2 neighborhood of $AB$. The length of  $L_1$ is at least  $4\pi + 4 \approx 16.566$, and this length can be realized by taking $L_1$ to be the boundary of the radius 2 neighborhood of $AB$ in the plane $z = 0$. To construct the other component, join the two points $A$ and $B$ by an arc $\alpha$ satisfying the following three conditions:
\begin{enumerate}
\item The  union of $\alpha$ with $AB$ forms a non-trivial knot contained in the half-space $z \ge 0$. 
\item  The arc $\alpha$ meets the plane $z = 0$ only at its endpoints and is perpendicular to the plane $z = 0$ at these points.
\item The union of $L_1$, $\alpha$ and the reflection  of $\alpha$ across the plane $z = 0$ forms a thick link. 
\end{enumerate}
The union of $\alpha$ and its reflection in the plane $z = 0$ is the second component $L_2$ of the thick link $L$. Figure \ref{link} shows an example of such a link.  

Theorem \ref{maintheorem} follows immediately from the following result, which gives an explicit lower bound on the length required for $L_1$, the unknotted component of $L$, if $L$ can be split by a physical isotopy.

\begin{theorem}\label{mainspecific}
If there is a physical isotopy of $L = L_1 \sqcup L_2$ that splits its two components, then the length of $L_1$ must be at least 
 $4\pi +  6 \approx 18.566$.
\end{theorem}

Since the link $L$ can be constructed with the length of  the unknotted component $L_1$ equal to   $4\pi + 4$,
this result implies Theorem \ref{maintheorem}.

The paper is arranged as follows. In Section 2 we give a lower bound on the boundary length of a non-positively curved  disk containing three disjoint disks of radius one.
In Section 3 we show that  if a family of disks spanning $L_1$ gives a homotopy from a disk in the $xy$-plane to a disk disjoint from $L_2$ and each disk in the homotopy intersects a neighborhood of $L_2$ in at most two components containing points of $L_2$, then $L_2$ is unknotted. In Section 4 we bring these results together to prove Theorem \ref{mainspecific}.

\section{An Isoperimetric Inequality}

To prove Theorem \ref{mainspecific}, we  show that if the unknotted component $L_1$ has length less than $4\pi +  6 $, then there are severe restrictions on how the other component can pass  through a natural spanning disc for $L_1$. This spanning disc, to be defined in Section \ref{bigproof}, is a cone with cone angle at least $2\pi$, and hence a $\textrm{CAT}(0)$ space. We are therefore led to finding a lower bound on the length of a curve in a $\textrm{CAT}(0)$ surface that bounds a disk enclosing three or more non-overlapping subdisks of radius one, as in Figure~\ref{threedots}.  

\begin{figure}[htbp] 
\centering
\psfrag{a}[][]{$f$}
\psfrag{b}[][]{$a$}
\psfrag{c}[][]{$b$}
\psfrag{d}[][]{$c$}
\psfrag{e}[][]{$d$}
\psfrag{f}[][]{$e$}
\psfrag{g}[][]{$c_1$}
\psfrag{h}[][]{$c_2$}
\psfrag{i}[][]{$c_3$}
\psfrag{x}[][]{$D$}
\psfrag{y}[][]{$\partial D$}
\includegraphics[width=2in]{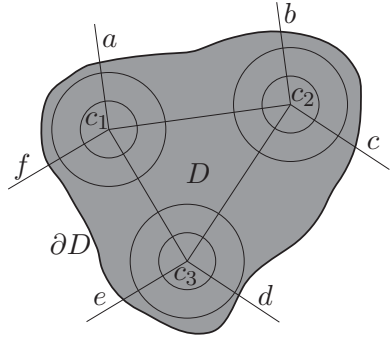} 
\caption{The boundary of this disk  has length at least $6 + 4\pi$.}
\label{threedots}
\end{figure}

\begin{proposition}\label{iso}
Let $P$ be a complete $CAT(0)$ surface and let  $D_1, D_2, D_3$ be three radius one subdisks of $P$ with disjoint interiors. Let $D \subset P$ be a disk containing $D_1, D_2, D_3$ such that
 $\partial D$ has distance at least 1 to any of  $D_1, D_2, D_3$.  Then the length of $\partial D$ is at least $4\pi +  6 $.
\end{proposition}
 
\begin{proof}
Let $T$ be the triangle with vertices at the center points  $c_1, c_2, c_3$ of the disks $D_1, D_2, D_3$.
Each edge of $T$ has length at least 2.
Let $a,b,c,d,e,f$ denote perpendicular rays from the sides of $T$ at its vertices, as in Figure~\ref{threedots}.
The curve  $\partial D$ intersects each of $a,b,c,d,e,f$ in at least one point.  We pick one such point for each, ordered cyclically around $\partial D$,
and refer to the intervening arcs  of $\partial D$ as the parts of $\partial D$ between them.
Since the edge of $T$  between $c_1$ and $c_2$ is perpendicular to $a$ and to $b$, it realizes the minimal distance of any path between them.
Thus the length of the part of $\partial D$ between $a$ and $b$ is at least 2. We can argue similarly for the  length of $\partial D$ between
$c$ and $d$, and between $e$ and $f$. Thus these three parts of $\partial D$  have total length at least  6.

The sum of the interior angles of $T$ is at most $\pi$, so the sum of the three angles between $f$ and $a$, between $b$ and $c$ and between $d$ and $e$ is at least $2\pi$.
Radial projection projects the remaining parts of $\partial D$ onto three circular arcs with total angle at least  $2\pi$. In a CAT(0) space, radius decreasing radial projection onto a circle of constant radius is length decreasing.  Since a circle of radius 2 has length at least  $4\pi$, it follows that the length of   $\partial D$ is at least  $6+ 4\pi$. 
\end{proof}
 
 {\bf Remark.}
A similar argument  shows that a curve enclosing two disks has length at least $4\pi +  4 $, and that the length of a curve enclosing $n>3$ disks is at least $4\pi +  2n $ if the centers of the disks form the vertices of a convex polygon. An argument for flat metrics was given in \cite{CKS2}.

\section{Sweepouts of Solid Tori}
 \label{sweep}
 
The following proposition gives a generalization of the fact that a 1-bridge knot is unknotted.
It  considers a generic 1-parameter family of disks, possibly singular, whose interiors sweep across a region containing a solid torus and concludes that if each disk meets the solid torus in at most two  components that cross its core, then the solid torus is unknotted.

\begin{proposition}\label{2dotted}
Let $T$ be a solid torus in $\RR^3$ with core $c$, such that both $T$ and $c$ are symmetric under reflection $r$ in the $xy$-plane. Suppose there is a homotopy of the disc $g_t:D  \rightarrow \RR^3 $, $t\in[-1,1]$, with the following properties:
\begin{enumerate}
\item  The curve $g_t(\partial D)$ is disjoint from $T$ for all  $t\in[-1,1]$.
\item The family of disks $g_t(D)$ is symmetric under reflection $r$ in the $xy$-plane; i.e. $g_0 (D)$ is contained in the $xy$-plane and $g_{-t} = r \circ g_{t}$.
\item The pre-image $g_0^{-1}(T)$ has two components, each containing a single point of $g_0^{-1}(c)$.
\item The disk $g_1(D)$ is disjoint from $c$.
\item For all $t \in[-1,1]$ the pre-image $g_t^{-1}(T) \subset  D$ has at most two components that contain a point of $g_t^{-1}(c)$.
\item The  map $g_t $  is generic with respect to the pair $(T,c)$.
\end{enumerate}
Then $c$ is unknotted.
\end{proposition}

Assumption (6) means that $g_t$ is transverse to $c$ with the exception of a finite number of times $t$ at which a birth or death of a pair of points of $g_t^{-1}(c)$ occurs, and that $g_t $  is transverse to $\partial T$ at these times. Additionally, $g_t$ is transverse to $\partial T$ except for a finite number of times at which $g_t^{-1}(\partial T)$ consists of finitely many simple closed curves and a single component that is either a wedge of finitely many circles (at a general saddle type singularity) or a point (at a birth or death singularity).

\begin{proof}[Proof of Proposition \ref{2dotted}]

To show that $c$ is unknotted, we will form a spanning disk $E$ for $c$ that is swept out by a continuous family of arcs in $\RR^3$ with endpoints on $c$ and with interiors disjoint from $c$. These arcs are of two types: 
The first type will lie on $g_t(D)$ and vary continuously with $t$
except at finitely many times $t$ when it jumps from one arc on $g_t(D)$ to another;  
the second type will interpolate continuously between the arcs just before and just after these jumps.

For a map $g:D\rightarrow\RR^3$ we call a point of $g^{-1}(c) \subset D $ a \emph{dot} and a component of $g^{-1}(T)$ that contains at least one point of $g^{-1}(c)$ a \emph{dotted component}.
 Thus $g_0^{-1}(T)$ contains two dotted components, each with a single dot.  A birth or death changes the number of points of $g_t^{-1}(c)$ in a component of $g_t^{-1}(T)$ by two, each intersecting with opposite orientation. 
Thus as long as there are two dotted components of  $g_t^{-1}(T)$  the number of dots in each remains algebraically $\pm 1$.

At time $t=0$, the pre-image $g_0^{-1}(c) \subset D$ consists of a pair of points, one in each dotted component.
Let $\alpha_0$ be an arc joining  these two points in $D$, with interior disjoint from $g_0^{-1}(c) \subset D$.
As $t$ increases, we take $\alpha_t$ to vary continuously through  arcs in $D$ joining dots in distinct dotted components with interiors disjoint from the dots. There is no obstruction to this while the collection of dots in $D$ is changing by an isotopy.  As long as there are two dotted components, there are two possible obstructions to the extension of $\alpha_t$:

\begin{enumerate}
\item  Part of $\alpha_t$ may run between two dots that come together and disappear in a death singularity.
\item One of the endpoints of $\alpha_t$ may disappear in a death singularity.

\end{enumerate}

In contrast, birth singularities do not pose a problem for the extension of the family of arcs  past the time at which they occur.

To avoid the  two problems above, we pick a small $\varepsilon >0 $ and at time $t_1' = t_1 - \varepsilon$ we jump from $\alpha_{t_1'}^- := \alpha_{t_1'}$ to a different arc $\alpha_{t_1'}^+$ that also joins points of $g_{{t_1'}}^{-1}(c) $ in distinct dotted components but that avoids a neighborhood of the death singularity. 
We will show how to construct $\alpha_{t_1'}^+$ so that this jump can be filled in appropriately for the construction of the disk $E$.
We will then extend the family of arcs $\alpha_t$ for $t > {t_1'}$ by starting with $\alpha_{t_1'}^+$ and continuing past $t_1$ until the next death singularity occurs at some time $t_2 > t_1$.

For a continuous map $g:D \rightarrow \RR^3$, we say that two arcs in the disk $D$ joining points of $g^{-1}(c)$ in distinct dotted components of  $g^{-1}(T)$ are {\em $g$-equivalent} if their images under  $g$ are homotopic  through arcs whose interiors are disjoint from $c$ and whose endpoints lie on $c$. Thus we seek to construct the arc  $\alpha_{t_1'}^+$ so that it is   $g_{t_1'}$-equivalent to $\alpha_{t_1'}^-$.

When  $g_{t}$ is transverse to $\partial T$ the dotted components of  $g_{t}^{-1}(T)$ form planar subsurfaces of $D$, each a disk with holes.
A boundary curve of a dotted component that has dots in both complementary pieces in $D$ is called \emph{primary}, and other boundary components are called {\em secondary}.

If $\alpha_{t_1'}^-$ leaves a dotted component $X$ through  a secondary boundary component $b$, it must re-enter $X$ through $b$, since $b$ is separating in $D$. 
Let $\beta$ be a sub-arc of $\alpha_{t_1'}^-$ running between two successive intersections of $\alpha_{t_1'}^-$ with $b$  and with interior outside of $X$. Then $\beta$ runs through either a dotless disc  or a dotless annulus in $D$. 
We can then homotope  $\beta$ into $b$ rel endpoints without crossing any dots. Push $\beta$ a little further into the interior of $X$.
Repeating this process we can homotope $\alpha_{t_1'}^-$ rel endpoints and without crossing dots  so that it crosses only primary boundary components. 
See Figure \ref{killnonsep}. We abuse notation somewhat and continue to refer to this arc as $\alpha_{t_1'}^-$.

\begin{figure}[h!] 
\centering
\psfrag{a}[][]{$\alpha_{t_1'}^-$}
\psfrag{b}[][]{$g_{t_1'}^{-1}(T)$}
\psfrag{c}[][]{$g_{t_1'}^{-1}(c)$}
\psfrag{d}[][]{$$}
\includegraphics[width=3.5in]{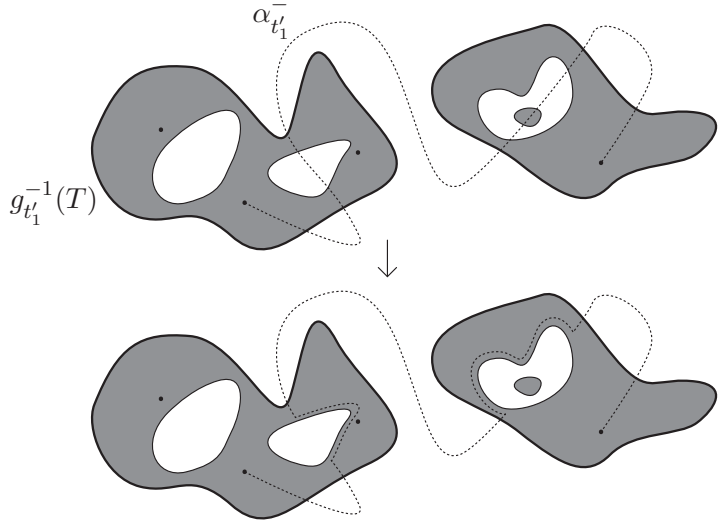} 
\caption{Removing intersections of $\alpha_{t_1'}^-$ with secondary boundary components of $g_{{t_1'}}^{-1}(T)$.}
\label{killnonsep}
\end{figure}

Our next goal is to arrange for $\alpha_{t_1'}^-$ to pass through each primary boundary component exactly once. We will achive this with the following lemma.

\begin{lemma}\label{spin}
Suppose $T$ is an embedded solid torus in $\RR^3$ with core $c$, and $g: D \rightarrow \RR^3$ is a continuous map of a disc into $\RR^3$ for which $g^{-1}(T)$ has two dotted components, each with image having algebraic intersection number with $c$  equal to $\pm1$. Let $\alpha$ be an arc in $D$ joining dots in distinct dotted components of $g^{-1}(T)$, and with interior disjoint from $g^{-1}(c)$.  
Let $\beta \subset g^{-1}(T)$ be a sub-arc arc of $\alpha$ that  starts and ends on the same primary boundary component  $b$  of $g^{-1}(T)$. Then there is an arc $\beta' \subset b$ with the same endpoints as $\beta$ and with the property that replacing $\beta$ with $\beta'$ in $\alpha$ yields an arc $\alpha'$ $g$-equivalent to $\alpha$. 

\end{lemma}
\begin{proof}
There is a homotopy of $\beta$ in $D$  rel endpoints to an arc $\overline{\beta}$ contained in $b$, possibly crossing dots.
So $g(\beta)$ is  homotopic  rel endpoints in $\RR^3$ to an arc $g(\overline{\beta})$ in $\partial T \cap g(D)$. This homotopy may pass outside $T$, as $\beta$ slides over holes of the dotted component. However the boundaries of these holes are secondary, and therefore have image under $g$ that is homotopically trivial  on $\partial T$.  Therefore $g(\beta)$ and $g(\overline{\beta})$ are homotopic rel endpoints in $T$.
The arc $g(\beta)$ is also homotopic rel endpoints in $T-c$, by radial projection away from $c$,  to an arc $\nu$ on $\partial T$.  
See Figure \ref{badarcs}, which for clarity shows only part of $g(D)$. 

\begin{figure}[h!] 
\centering
\psfrag{a}[][]{$g(\beta)$}
\psfrag{b}[][]{$g(\overline{\beta})$}
\psfrag{c}[][]{$\nu$}
\includegraphics[width=2.5in]{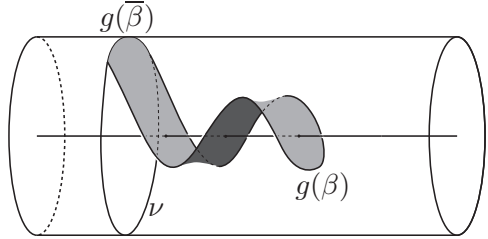} 
\caption{Two ways to homotope $g(\beta)$ onto $\partial T$.}
\label{badarcs}
\end{figure}

Now, $g(\overline{\beta})$ and $\nu$ are homotopic  rel endpoints in $T$ and so in $\partial T$ they differ by a multiple of a meridian.
Note that the curve  $g(b)$ is a meridian, since it bounds a disk in $T$ meeting $c$ algebraically once.
So $g(\beta)$ can be homotoped  rel endpoints in the complement of $c$ to  $\nu$, and then in turn to  a curve  formed by concatenating $g(\overline{\beta})$ with a multiple of  $g(b)$. Take $\beta'$ to be $\overline{\beta}$ followed by this multiple of $b$. 
\end{proof}

Now suppose that  $\beta$ is a sub-arc  of $\alpha_{t_1'}^-$ that enters and leaves a dotted component. Using Lemma \ref{spin} we can replace it with a $g_{{t_1'}}$-equivalent arc $\beta'$ 
that lies entirely on $g_{{t_1'}}^{-1}(\partial T)$, and then perturb $\beta'$ slightly so that it is disjoint from $g_{{t_1'}}^{-1}(T)$, as illustrated in Figure \ref{lemmaspin}.  
In this way we replace $\alpha_{t_1'}^-$ with a $g_{{t_1'}}$-equivalent  arc that has fewer intersections with dotted components, and
by repeating we may remove all  sub-arcs of $\alpha_{t_1'}^-$ that enter and leave a dotted component. 
We continue to refer to the resulting arc as $\alpha_{t_1'}^-$. Note that $\alpha_{t_1'}^-$ may now intersect itself.

\begin{figure}[h!] 
\centering
\psfrag{b}[][]{$\beta$}
\psfrag{c}[][]{$g_{t_1-\varepsilon}^{-1}(c)$}
\psfrag{d}[][]{$$}
\includegraphics[width=4in]{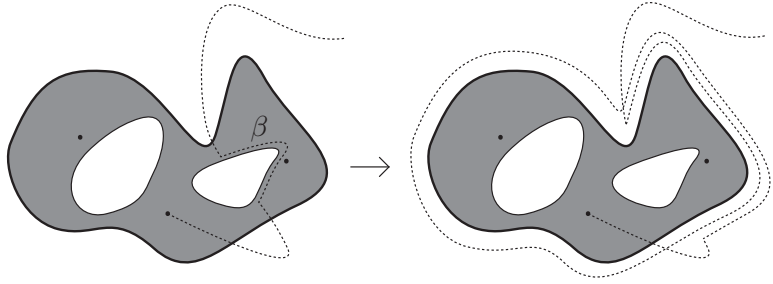} 
\caption{Removing an arc in $D$ that that starts and ends on the same primary boundary component of $g^{-1}(T)$.}
\label{lemmaspin}
\end{figure}

We have found an arc $g_{{t_1'}}$-equivalent to the original arc $\alpha_{t_1'}^-$ that starts at a dot in one dotted component, exits that dotted component, then enters the second dotted component and finally ends at a dot. 
The following lemma allows us to find a $g_{{t_1'}}$-equivalent arc that replaces a sub-arc running from a primary boundary component to a dot with any other arc running from the same point to a dot and not leaving $g_{{t_1'}}^{-1}(T)$. 

We define an arc in a solid torus $T$ with core $c$ to be {\em proper} if it has one endpoint on $\partial T$, the other endpoint on $c$, and  interior  disjoint from $c$.

\begin{lemma} \label{fixends}
Let $T$ be a solid torus with core $c$. Let $\gamma$ and $\gamma'$ be proper arcs in $T$ with $\gamma \cap \partial T= \gamma' \cap \partial T$. Then $\gamma$ and $\gamma'$ can be joined by a homotopy of proper arcs, keeping the endpoint on $\partial T$ fixed. 
\end{lemma}
\begin{proof}
We can lift $\gamma$ and $\gamma'$ to the universal cover of $T$, which is homeomorphic to  $\textrm{(disk)}\times \RR$,  so that their common endpoint on $\partial T$ lifts to the same point $x$, and $c$ lifts to $\{0\} \times \RR$.  Each lift can be homotoped, keeping $x$ fixed and moving points only along the $\RR$-factor, to the slice $\textrm{(disk)}\times \textrm{\{point\}}$ containing $x$.  Further, the resulting arcs are homotopic rel endpoints in $\textrm{(disk)}\times \textrm{\{point\}}$ via arcs that miss $\{0\} \times \textrm{\{point\}}$ in their interior. Therefore the lifts of $\gamma$ and $\gamma'$ are homotopic through arcs joining $x$ to $\{0\} \times \RR$ and with interiors disjoint from $\{0\} \times \RR$.  The projection of this homotopy to $T$  gives a homotopy joining $\gamma$ and $\gamma'$ through proper arcs in $T$. 
\end{proof}
 
Now suppose  that a death singularity takes place in the dotted component containing the initial segment of $\alpha_{t_1'}^-$. Let $\gamma_1$ be the initial segment of $\alpha_{t_1'}^-$, running from a dot to the primary boundary component. Choose an arc $\gamma_1'$ in the same dotted component that runs from a dot to the point where  $\gamma_1$ exits the dotted component, so that $\gamma_1'$ is disjoint from a neighborhood containing the two dying dots. This is possible because the number of dots in each dotted component is odd.  By Lemma \ref{fixends}, $\alpha_{t_1'}^-$ is $g_{{t_1'}}$-equivalent to the arc formed by replacing $\gamma_1$ by $\gamma_1'$.  If the death singularity takes place in the other dotted component then a similar change may be made.  We take $\alpha_{t_1'}^+$ to be the arc, $g_{{t_1'}}$-equivalent to $\alpha_{t_1'}^-$, which is obtained from $\alpha_{t_1'}^-$ after making all these changes. See Figure \ref{ends}.

 \begin{figure}[h!] 
\centering

\psfrag{a}[][]{$\gamma_1$}
\psfrag{b}[][]{$\gamma_1'$}
\psfrag{d}[][]{$$}
\includegraphics[width=4in]{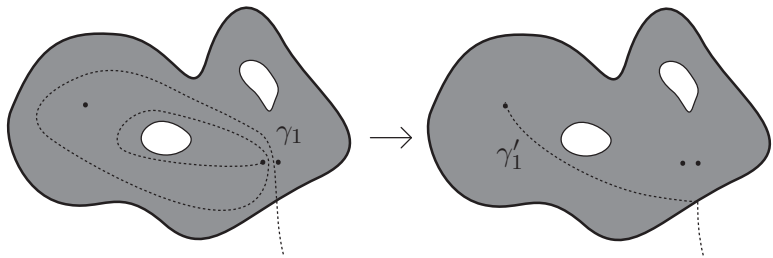} 
\caption{Changing $\alpha_{t_1'}^-$ so that it avoids  a death singularity.}
\label{ends}
\end{figure}

We have constructed a  family of arcs $\alpha_t$ in $D$ that varies continuously until time ${t_1'}= t_1-\varepsilon$. It then jumps from $\alpha_{t_1'}^-$ to the $g_{{t_1'}}$-equivalent 
 arc $\alpha_{t_1'}^+$.
 The arc $\alpha_{t_1'}^+$ was chosen to avoid a neighborhood of the death singularity at time $t_1$ so we can extend the family of arcs $\alpha_t$ past time $t=t_1$ until just before the next death singularity at time $t=t_2$. We then repeat this process. 

Eventually, at time $t=t_l\in(0,1)$ the two dotted components merge along ${g_{t_l}}^{-1}(\partial T)$. At this time ${g_{t_l}}^{-1} (\partial T)$ 
consists of a collection of finitely many simple closed curves and one wedge of finitely many circles embedded in $D$.
The arc $\alpha_{t_l}$  begins and ends at a dot, and may run in and out of the single dotted component.  

We now look at the complementary components in $D$ of the single dotted component of ${g_{t_l}}^{-1}(T)$.  Each of these is either a disk or an annulus with $\partial D$ as one boundary component. Any subarc $\beta$ of $\alpha_{t_l}$ that runs out of a dotted component  into a complementary component $X$ eventually leaves $X$ and reenters the dotted component.  
Since $X$ is either a disk or an  annulus with $\partial D$ as one boundary component, we can homotope $\beta$ rel endpoints off $X$ and into the dotted component without passing through any dots, since all dots lie within the single dotted component.
In this way we can homotope $\alpha_{t_l}$ so that it lies entirely within the dotted component of ${g_{t_l}}^{-1}(T)$.
This means that $ g_{t_l} (\alpha_{t_l})$ now lies entirely within $T$.  It is then straightforward to shrink  $ g_{t_l} (\alpha_{t_l})$ within $T$, keeping its interior disjoint from $c$ and its endpoints on $c$, until it collapses to a point on $c$. 
 
We now form a  spanning disk $E$ for $c$. Let $\beta_t = g_t \circ \alpha_t, ~t\in [0, t_l]$. Then $\beta_t$ is a family of arcs in $\RR^3$ whose endpoints lie on $c$ and whose interiors are disjoint from $c$. These arcs vary continuously except at finitely many times $t_1', t_2',   \dots, t_n'$ just before death singularities.  At these times the limiting arcs $\alpha_{t_i'}^-$ and $\alpha_{t_i'}^+$  as $t$ approaches $t_i'$ from below and above are $g_{t_i'}$-equivalent. 
Finally, $ \alpha_{t_l}$ is  homotopic to a point on $c$ via arcs that start and end on $c$ but have interiors disjoint from $c$. Therefore there is a family of arcs with endpoints on $c$ sweeping out a disk with interior in the complement of $c$ that represents a homotopy of 
 $\beta_0$ to a point in $c$.
Let $\overline{E}$ denote the union of these arcs in $\RR^3$ and  take $E$ to be the disk obtained by taking the union of $\overline{E}$ with its reflection  in the $xy$-plane. Note that the interior of $E$ does not intersect $c$ and that $\partial E \subset c$.

Let $a$ be  one of the two points of intersection of $ c$ with the $g_0 (D)$. Then $ \partial E$ intersects $a$ in an odd number of points,
one coming from  $ \beta_0 $ 
and an additional even number coming from equal numbers of intersections  of $a$ with $\partial \overline{E}$  and its reflection. 
So $ \partial E$ is non-trivial in $\pi_1{(c)}$. By the Loop Theorem \cite{Papakyriakopoulos}, $c$  is the boundary of an embedded disc and therefore unknotted.
\end{proof}

\section{Proof of Theorem~\ref{mainspecific}} \label{bigproof}

We now prove Theorem~\ref{mainspecific}, showing that if $L$ can be split via a physical isotopy then the length of $L_1$ must be at least $4\pi + 6$.

\begin{proof}[Proof of Theorem~\ref{mainspecific}]

Assume there is a physical isotopy $I_s$, $s \in [0,1]$, of $\RR^3$  with $I_0$ the identity, $ I_1$  taking $L_1$ and $L_2$ to opposite sides of a plane,  and with the length of the unknotted component $L_1$  being less than $4\pi + 6$. We will derive a contradiction.

During the course of the isotopy $I_s$ it is possible that the radius one solid torus neighborhoods of the two link components bump against themselves or each other.  
We describe a  slight modification of the isotopy that  keeps the two components embedded and disjoint.
Throughout the isotopy a neighborhood of radius $r<1$ describes two embedded disjoint solid torus neighborhoods of each of  $L_1 $ and $L_2$.
Take $r=1-\varepsilon'$ to be slightly less than one and then rescale the entire isotopy $I_s$ by $1/(1-\varepsilon')$.  
This restores the radius to 1 at the cost of slightly lengthening $L_1$ and $L_2$.  With $\varepsilon' $ small, the length of $L_1$ remains below $4\pi + 6$. The rescaled physical isotopy is then $(1+\varepsilon$)-thick, with $\varepsilon = \varepsilon'/(1-\varepsilon')$.

For a curve $c$ in $\RR^3$, let $T(c)$ denote the radius-1 neighborhood of $c$. Without loss of generality take $I_s$ so that $I_s(T(L_i)) = T(I_s(L_i))$ for $i=1,2$, so that $I_s$ respects radius 1 neighborhoods of $L_1$ and $L_2$. Let $T_s$ be the solid torus $T(I_s(L_2))$. For each $s \in [0,1]$ let $x_s$ be the center of mass of the  embedded  curve $I_s(L_1)$ 
and let $f_s: D \to \RR^3$ parametrize the disk forming the cone over  $I_s(L_1)$ with cone point $x_s$. Since the cone point is inside the convex hull, its cone angle is always  at least $2\pi$ \cite{Gage, Gromov}. 
Moreover each disc is flat  except  at the cone point and therefore a subdisk of a complete  CAT(0) surface obtained by extending the rays from the cone point to infinity.

Now perturb $f_s$, $s \in[0,1]$, so that the family of maps $f_s$ is generic, but keeping $f_s$ fixed for $s$ in a small neighborhood of $0$ and fixed on ${\partial D}$ for all $s$.
By generic, we mean that:
\begin{enumerate}
\item $f_s$ is transverse to $c$ except for a finite number of times $s$ at which a birth or death of a pair of points of $f_s^{-1}(c)$ occurs;
\item$f_s$  is transverse to $\partial T$ at these times; and
\item $f_s$ is transverse to $\partial T$ except for a finite number of times when $f_s^{-1}(\partial T)$ consists of finitely many simple closed curves and a single component that is a wedge of finitely many circles (in the case of a saddle type singularity) or a single point (in the case of a birth or death singularity).
\end{enumerate}
Genericity can be achieved by approximating the appropriate parts of $f_s$ by PL maps and using general position. The perturbation can be made arbitrarily $\mathcal{C}^0$-small, and we let $f_s'$ denote the result of perturbing $f_s$ in this manner. 

Each component of $f_s'^{-1}(T_s)$ in $D$ that contains a point of  $f_s'^{-1}(I_s(L_2))$ contains a disc of radius one in $D$ enclosing that point. The distance in $D$ of $\partial D$  from each of these components is at least one. Suppose for a contradiction that  there are three components of $f_s'^{-1}(T_s)$  containing a point of $f_s'^{-1}(I_s(L_2))$ for some $s$ and furthermore suppose that this is true no matter how small we made the perturbation of $f_s$ that gave $f_s'$. Then $f_s^{-1}(T_s)$  contains three disks of radius 1 with disjoint interiors, and with $\partial D$ having distance at least one from each disk. This contradicts Proposition \ref{iso}, since $L_1$ has length less than $4\pi + 6$. Hence, by taking the perturbation to obtain $f_s'$ to be sufficiently small,  we can arrange that  for all $s$ there are at most two components of $f_s'^{-1}(T_s)$  containing a point of $f_s'^{-1}(I_s(L_2))$.

Define a family of disks  $h_s : D \to  \RR^3$, $0 \le s \le 1$, by setting $h_s = I_s^{-1} \circ f_s'$. Extend $h_s$ to $-1 \le s \le 0$ by  reflecting through the $xy$-plane, setting $h_{s} = r \circ h_{-s} $ for $s<0$. Each  $h_s$ maps $\partial D$ to $L_1$ and for $s\in[-1,1]$ the pre-image  $h_s^{-1} (T(L_2))$ has at most two components containing a point of $h_s^{-1}(L_2) $. Moreover  $h_0(D)$  lies in the $xy$-plane and $h_1(D)$ and   $h_{-1}(D)$ are disjoint from $L_2$. The disks $h_s(D)$ now satisfy all the conditions of Proposition \ref{2dotted}, implying that $L_2$ is unknotted. This contradiction proves Theorem \ref{mainspecific}. \end{proof}

\bibliographystyle{amsplain}
\bibliography{gordrefs}
 
\end{document}